\documentclass{article}
\usepackage[utf8]{inputenc}
\usepackage{amsmath,amsthm,amssymb,amsfonts}
\usepackage{verbatim}
\usepackage{array}

\usepackage{hyperref}

\usepackage{stmaryrd} 
\usepackage{hyperref}

\usepackage[colorinlistoftodos]{todonotes}

\newtheorem{thm}{Theorem}
\newtheorem{lem}[thm]{Lemma}
\newtheorem{prop}[thm]{Proposition}
\newtheorem{exl}[thm]{Example}
\newtheorem{cor}[thm]{Corollary}

\newtheorem{defn}[thm]{Definition}
\theoremstyle{remark}
\newtheorem{remark}[thm]{Remark}
\theoremstyle{defn}

\numberwithin{thm}{section}

\newcommand{\adj}{\leftrightarrow}
\newcommand{\adjeq}{\leftrightarroweq}

\DeclareMathOperator{\id}{id}
\def\N{{\mathbb N}}

\DeclareMathOperator{\Fix}{Fix}
\newcommand{\Z}{\mathbb{Z}}

\title{Limiting Sets in Digital Topology}
\author{Laurence Boxer 
        \thanks{Department of Computer and Information Sciences,
        Niagara University, NY 14109, USA; and \newline
        Department of Computer Science and Engineering,
        State University of New York at Buffalo\newline
        email: boxer@niagara.edu}
}
\date{}

\begin{document}

\maketitle

\begin{abstract}
   Freezing sets and cold sets have been introduced 
   as part of the theory of fixed points in digital 
   topology. In this paper, we introduce a generalization of these notions, the {\em limiting set},
   and examine properties of limiting sets.
   

\end{abstract}

\section{Introduction}
The study of freezing sets and cold sets (see
Definitions~\ref{freezeDef} and~\ref{s-cold-def})
stems from the fixed point theory of digital topology.
Freezing and cold sets were studied in papers
including~\cite{BxFpSets2,BxConvex,subsetsAnd,arbDim,consequences,coldAndFreezing}.

For a continuous function $f: (X,\kappa) \to (X,\kappa)$
and $A \subset X$, weaker restrictions on
$(A,f|_A)$ than appear in
the definitions of freezing and cold sets 
may yet result in interesting restrictions on $f$.
In the current paper, we study how restrictions
on how far $f$ can move any member of $A$,
can powerfully restrict how far $f$ can 
move any member of $X$.

\section{Preliminaries}
We use $\Z$ for the set of integers, $\N$ for the
set of natural numbers, and $\N^*$ for the set of
nonnegative integers.
Given $a,b \in \Z$, $a \le b$,
\[ [a,b]_{\Z} = \{ \, z \in \Z \mid a \le z \le b \, \}.
\]

For a finite set $X$, the notation $\#X$ represents 
the number of distinct members of $X$.

Given a function $f: X \to Y$ and $A \subset X$,
the notation $f|_A$ indicates the restriction of $f$
to $A$. Also, we will occasionally use the notation
$fx$ to abbreviate $f(x)$.

\subsection{Adjacencies}
Much of this section is quoted or paraphrased from~\cite{bs20}.

A digital image is a pair $(X,\kappa)$ where
$X \subset \Z^n$ for some $n$ and $\kappa$ is
an adjacency on $X$. Thus, $(X,\kappa)$ is a graph
with $X$ for the vertex set and $\kappa$ 
determining the edge set. Usually, $X$ is finite,
although there are papers that consider infinite $X$,
e.g., for covering spaces; we will consider infinite~$X$
in section~\ref{infX}. 
Usually, adjacency reflects some type of
``closeness" in $\Z^n$ of the adjacent points.
When these ``usual" conditions are satisfied, one
may consider the digital image as a model of a
black-and-white ``real world" digital image in which
the black points (foreground) are 
the members of $X$ and the white points 
(background) are members of $\Z^n \setminus X$.

We write $x \adj_{\kappa} y$, or $x \adj y$ when
$\kappa$ is understood or when it is unnecessary to
mention $\kappa$, to indicate that $x$ and
$y$ are $\kappa$-adjacent. Notations 
$x \adjeq_{\kappa} y$, or $x \adjeq y$ when
$\kappa$ is understood, indicate that 
$x \adj_{\kappa} y$ or $x = y$.

The most commonly used adjacencies are the
$c_u$ adjacencies, defined as follows.
Let $X \subset \Z^n$ and let $u \in \N$,
$1 \le u \le n$. Then for points
\[x=(x_1, \ldots, x_n) \neq (y_1,\ldots,y_n)=y
  \mbox{ in } X
\]
we have $x \adj_{c_u} y$ if and only if
\begin{itemize}
    \item for at most $u$ indices $i$ we have
          $|x_i - y_i| = 1$, and
    \item for all indices $j$, $|x_j - y_j| \neq 1$
          implies $x_j=y_j$.
\end{itemize}

In low dimensions, the $c_u$-adjacencies are often
denoted by the number of adjacent points a point 
can have in the adjacency. E.g.,
\begin{itemize}
\item in $\Z$, $c_1$-adjacency is 2-adjacency;
\item in $\Z^2$, $c_1$-adjacency is 4-adjacency and
      $c_2$-adjacency is 8-adjacency;
\item in $\Z^3$, $c_1$-adjacency is 6-adjacency,
      $c_2$-adjacency is 18-adjacency, and 
      $c_3$-adjacency is 26-adjacency.
\end{itemize}

In this paper, we mostly use the $c_1$ and $c_n$ adjacencies.

When $(X,\kappa)$ is understood to be a digital image under discussion,
we use the following notations. For $x \in X$,
\[ N(x) = \{  \, y \in X \,| \, y \adj_{\kappa} x \, \},
\]
\[ N^*(x) = \{ \, y \in X \,| \, y \adjeq_{\kappa} x \, \} = N(x) \cup \{ \, x \, \}.
\]

\begin{defn}
{\rm \cite{Ros79}}
Let $X \subset \Z^n$. The {\em boundary of } $X$ is
\[ Bd(X) = \{ \, x \in X \mid \mbox{there exists } y \in \Z^n \setminus X
              \mbox{ such that } x \adj_{c_1} y \, \}.
\]
\end{defn}

\begin{defn} 
\label{pathDef}
A {\em digital $\kappa$-path} $P$ in $(X,\kappa)$ from
$x \in X$ to $y \in X$ of {\em length} = $n$ is
a subset $P=\{ \, x_i \, \}_{i=0}^n$ of $X$ such that $x_0=x$, $x_{i-1} 
\adj_{\kappa} x_i$ for 
$i = 1, \ldots, n$, and
$x_n=y$.
\end{defn}

\subsection{Digitally continuous functions}
Much of this section is quoted or paraphrased from~\cite{bs20}.

We denote by $\id$ or $\id_X$ the
identity map $\id(x)=x$ for all $x \in X$.

\begin{defn}
{\rm \cite{Rosenfeld, Bx99}}
Let $(X,\kappa)$ and $(Y,\lambda)$ be digital
images. A function $f: X \to Y$ is 
{\em $(\kappa,\lambda)$-continuous}, or
{\em digitally continuous} or just {\em continuous} when $\kappa$ and
$\lambda$ are understood, if for every
$\kappa$-connected subset $X'$ of $X$,
$f(X')$ is a $\lambda$-connected subset of $Y$.
If $(X,\kappa)=(Y,\lambda)$, we say a function
is {\em $\kappa$-continuous} to abbreviate
``$(\kappa,\kappa)$-continuous."
\end{defn}

Similar notions are referred to as {\em immersions}, 
{\em gradually varied operators}, and {\em gradually varied mappings}
in~\cite{Chen94,Chen04}.

\begin{thm}
{\rm \cite{Bx99}}
\label{contThm}
A function $f: X \to Y$ between digital images
$(X,\kappa)$ and $(Y,\lambda)$ is
$(\kappa,\lambda)$-continuous if and only if for
every $x,y \in X$, if $x \adj_{\kappa} y$ then
$f(x) \adjeq_{\lambda} f(y)$.
\end{thm}

\begin{thm}
\label{composition}
{\rm \cite{Bx99}}
Let $f: (X, \kappa) \to (Y, \lambda)$ and
$g: (Y, \lambda) \to (Z, \mu)$ be continuous 
functions between digital images. Then
$g \circ f: (X, \kappa) \to (Z, \mu)$ is continuous.
\end{thm}

A function $f: (X,\kappa) \to (Y,\lambda)$ is
an {\em isomorphism} (called a {\em homeomorphism}
in~\cite{Bx94}) if $f$ is a continuous bijection
such that $f^{-1}$ is continuous.

For $X \in \Z^n$, the {\em projection to the $i^{th}$ coordinate}
is the function $p_i: X \to \Z$ defined by
\[ p_i(x_1, \ldots, x_n) = x_i.
\]

We denote by $C(X,\kappa)$ the set of $\kappa$-continuous
functions $f: X \to X$.

\begin{defn}
{\rm (\cite{Bx99}; see also \cite{Khalimsky})}
\label{htpy-2nd-def}
Let $X$ and $Y$ be digital images.
Let $f,g: X \rightarrow Y$ be $(\kappa,\kappa')$-continuous functions.
Suppose there is a positive integer $m$ and a function
$h: X \times [0,m]_{\Z} \rightarrow Y$
such that

\begin{itemize}
\item for all $x \in X$, $h(x,0) = f(x)$ and $h(x,m) = g(x)$;
\item for all $x \in X$, the induced function
      $h_x: [0,m]_{\Z} \rightarrow Y$ defined by
          \[ h_x(t) ~=~ h(x,t) \mbox{ for all } t \in [0,m]_{\Z} \]
          is $(c_1,\kappa')-$continuous. Thus, $h_x$ is a path in $Y$.
\item for all $t \in [0,m]_{\Z}$, the induced function
         $h_t: X \rightarrow Y$ defined by
          \[ h_t(x) ~=~ h(x,t) \mbox{ for all } x \in  X \]
          is $(\kappa,\kappa')-$continuous.
\end{itemize}
Then $h$ is a {\em digital $(\kappa,\kappa')-$homotopy between} $f$ and
$g$, and $f$ and $g$ are {\em digitally $(\kappa,\kappa')-$homotopic in} $Y$.
\end{defn}

If $f \in C(X,\kappa)$ such that $f|_{f(X)} = \id_{f(X)}$,
then $f$ is a {\em $(\kappa-)$retraction of $X$}
and $f(X)$ is a {\em $(\kappa-)$retract of $X$}.

\subsection{Digital paths and simple closed curves}

By Definition~\ref{pathDef} and
Theorem~\ref{contThm},
$\{ \, x_i \, \}_{i=0}^n  \subset X$  is a 
$\kappa$-path from $x$ to $y$ if and only if the function
$f: [0,n]_{\Z} \to X$ given by $f(i) = x_i$ 
is $(c_1,\kappa)$-continuous. Such a 
function $f$ is also
called a {\em $\kappa$-path} 
from $x$ to $y$.

If $i \neq j$ implies $x_i \neq x_j$ in a digital path $P$ in $(X,\kappa)$,
we say $P$ is an {\em arc} from $x \in X$ to $y \in X$.

\begin{defn}
{\rm \cite{Han05}}
Let $(X,\kappa)$ be a connected digital image. The 
{\em shortest path metric} for $(X,\kappa)$ is
\[ d_{(X,\kappa)}(x,y) = \min\{length(P) ~|~ P 
   \mbox{ is a $\kappa$-arc in $X$ from $x$ to $y$}\},
   \mbox{ for } x,y \in X.
\]
\end{defn}

Let
\[ N^*(x,m) = \{ \, y \in X \mid d_{(X,\kappa)}(x,y) \le m
            \, \}
\]
Notice $N^*(x,1) = N^*(x)$.

Let $(X,\kappa)$ be a finite connected digital image. 
The {\em diameter of } $(X,\kappa)$ is
\[ diam(X,\kappa) = \max \{ \, d_{(X,\kappa)}(x,y) \mid
     x,y \in X \, \}.
\]

A {\em digital simple closed curve} of $n$ points is a digital image
$(C_n,\kappa)$ such that $C_n = \{ \, c_i \, \}_{i=0}^{n-1}$,
where $c_i \adj_{\kappa} c_j$ if and only if $i = (j \pm 1) \mod n$.
An indexing that satisfies these properties is a 
{\em circular ordering} or {\em circular indexing}.
Often, we require $n \ge 4$. $C_n$ may also be
called a {\em cycle on $n$ points}.

Let $q$ and $q'$ be distinct members of $C_n$. 
These points determine distinct arcs in $C_n$ from $q$
to $q'$. If one of these arcs is shorter than the 
other, the former is the
{\em unique shortest arc in $C_n$ from $q$ to $q'$}.

\subsection{Freezing sets}
\label{freezeDefSec}
\begin{defn}
{\rm \cite{BxFpSets2}}
\label{freezeDef}
Let $(X,\kappa)$ be a digital image. We say
$A \subset X$ is a {\em freezing set for $(X,\kappa)$}
if given $f \in C(X,\kappa)$, $A \subset \Fix(f)$ implies
$f=\id_X$. We say a freezing set $A$ is {\em minimal} if no
proper subset of $A$ is a freezing set for $(X,\kappa)$.
\end{defn}

We recall the following.

\begin{prop}
{\rm \cite{bs20}}
\label{uniqueShortest}
Let $(X,\kappa)$ be a digital
image and $f \in C(X,\kappa)$.
Suppose $x,x' \in \Fix(f)$ are
such that there is a unique
shortest $\kappa$-path $P$ in $X$ 
from $x$ to $x'$. Then
$P \subset \Fix(f)$.
\end{prop}

\begin{thm}
{\rm \cite{BxFpSets2}}; corrected proof in~{\rm \cite{arbDim}}
\label{cornersFreeze}
Let $X = \Pi_{i=1}^n [r_i,s_i]_{\Z}$.
Let $A = \Pi_{i=1}^n \{r_i,s_i\}$.
\begin{itemize}
\item Let $Y = \Pi_{i=1}^n [a_i,b_i]_{\Z}$ be
      such that $[r_i,s_i] \subset [a_i,b_i]_{\Z}$ for all $i$. 
      Let $f: X \to Y$ be $c_1$-continuous. If 
      $A \subset \Fix(f)$, then $X \subset \Fix(f)$.
\item $A$ is a freezing set for $(X,c_1)$ that is minimal for 
      $n \in \{1,2\}$.
\end{itemize}
\end{thm}

\begin{thm}
\label{3ptsForCycles}
{\rm \cite{BxFpSets2}}
Let $n > 4$. Let $x_i,x_j,x_k$ be
distinct members of $C_n$ be such
that $C_n$ is a union of unique shorter arcs
determined by pairs of these points.
Let $f \in C(C_n,\kappa)$.
Then $f=\id_{C_n}$ if and only if
$\{x_i,x_j,x_k\} \subset \Fix(f)$; i.e.,
$\{x_i,x_j,x_k\}$ is a freezing set for $C_n$.
Further, this freezing set is minimal.
\end{thm}

\section{$n$-limited sets}
In this section, we introduce limiting sets and
explore some of their basic properties.

\subsection{Definition and general properties}
\label{defAndGeneralProps}
Freezing sets and $s$-cold sets (the latter defined 
below) motivate this work.

\begin{defn}
\label{s-cold-def}
{\rm \cite{BxFpSets2}}
Given $s \in \N^*$, we say $A \subset X$ is an
{\em $s$-cold set} for the connected digital image $(X,\kappa)$
if given $g \in C(X,\kappa)$ such that
$g|_A = \id_A$, then for all $x \in X$,
$d_{(X,\kappa)}(x,g(x)) \le s$.
A {\em cold set} is a $1$-cold set.
\end{defn}

The notion of an $s$-cold set 
generalizes that of the freezing set, since a freezing set is 0-cold.
Next, we introduce a generalization of a cold set.

\begin{defn}
\label{n-map}
Let $X$ be a $\kappa$-connected digital image. 
Let $f \in C(X,\kappa)$ 
and let $A \subset X$. Let $m,n \in \N^*$.
If for all $x \in A$ we have 
$d_{(X,\kappa)}(x,f(x)) \le n$, we say
$f|_A$ is an $n$-map. We say $f$ is an $n$-map 
if $f|_X$ is an $n$-map.
If for all $f \in C(X,\kappa)$, $f|_A$ being 
an $m$-map implies $f$ is an $n$-map, then $A$ is an
{\em $(m,n)$-limiting set for} $(X,\kappa)$
and $(X,\kappa)$ is {\em $(A,m,n)$-limited}.
Such a set is a {\em minimal $(m,n)$-limiting set for} $(X,\kappa)$ if no proper subset $A'$ of $A$ is an
$(m,n)$-limiting set for $(X,\kappa)$.
\end{defn}

\begin{remark}
\label{elementaryProps}
The following are easily observed:
\begin{itemize}
    \item $f \in C(X,\kappa)$ is a 0-map if and only if $f = \id_X$.
          Therefore, by Definitions~\ref{freezeDef} and~\ref{n-map}, $(X,\kappa)$
          is $(A,0,0)$-limited if and only if
          $A$ is a freezing set for $(X,\kappa)$.
    \item More generally, $(X,\kappa)$
          is $(A,0,s)$-limited if and only if
          $A$ is an $s$-cold set for $(X,\kappa)$.
    \item $f \in C(X,\kappa)$ is a 1-map if and only if every
              $x \in X$ is an approximate fixed point of $f$, i.e., 
              $f(x) \adjeq_{\kappa} x$. Therefore, $(X,\kappa)$ is
              $(A,1,1)$-limited if and only if for every
              $f \in C(X,\kappa)$, $f|_A$ being a $1$-map implies
              every $x \in X$ is an approximate fixed point of $f$.
    \item If $X$ is finite and $\kappa$-connected, 
          then $X$ is $(A,m,diam(X,\kappa))$-limited
          for $0 \le m \le diam(X,\kappa)$ and every nonempty $A \subset X$.
    \item If $f \in C(X,\kappa)$ such that $f|_A$ is an
          $m$-map and $A' \subset A$, then $f|_{A'}$ is
          an $m$-map.
    \item If $(X,\kappa)$ is $(A,m_0,n_0)$-limited, 
          $0 \le m_1 \le m_0$, and $n_0 \le n_1$, then
          $(X,\kappa)$ is $(A,m_1,n_1)$-limited.
    \item If $f_i \in C(X,\kappa)$ are, respectively,
          $n_i$ maps, $i \in \{ \, 1, 2 \, \}$, then
          $f_2 \circ f_1$ is an $(n_1 + n_2)$-map.
\end{itemize}
\end{remark}

\begin{exl}
$\{0\}$ is a minimal $(0,1)$-limiting set for 
$([0,1]_{\Z}, c_1)$, but is not $(0,0)$-limiting.
\end{exl}

\begin{proof}
Both assertions follow from the observation that if
$f: [0,1]_{\Z} \to [0,1]_{\Z}$ is the function $f(x) = 0$, then
$f \in C([0,1]_{\Z}, c_1)$.
\end{proof}

Perhaps the significance of the property of being $(A,m,n)$-limited can be
understood as follows. If $n$ is much smaller than $diam(X,\kappa)$ and
$f|_A$ is an $m$-map, then $(X,\kappa)$ being $(A,m,n)$-limited implies
$f$ is an $n$-map, so $f$ does not move any
point of $X$ by very much. Perhaps as a consequence,
$X$ and $f(X)$ will resemble each other, although 
such a conclusion will admit subjective exceptions.

The following generalizes the fact (that one can
deduce from the second part of Theorem~\ref{cornersFreeze})
that the set of endpoints of a digital interval is
a freezing set for the interval.

\begin{prop}
\label{intervalLimitation}
Let $X = [a,b]_{\Z}$. Let $m \in \N^*$. Let
$A = \{ \, a,b \, \}$. Then $(X,c_1)$ is
$(A, m, m)$-limited. Further, $A$ is minimal if and
only if $b-a > m$.
\end{prop}

\begin{proof}
Let $f \in C(X,c_1)$ such that $f|_A$ is an $m$-map. Let 
$x \in X$, $a < x < b$.
\begin{itemize}
    \item Suppose $f(x) > x + m$. Since $f(a) \le a + m$
          we have
          \[ f(x) - f(a) > x + m - (a+m) = x - a,
          \]
          which is contrary to the continuity of $f$.
    \item We obtain a similar contradiction if 
          $f(x) < x - m$.
\end{itemize}
Therefore, we must have $\mid f(x) - x \mid \le m$, so
$f$ is an $m$-map.

Let $f_a, f_b: X \to X$ be given by $f_a(x) = a$, 
$f_b(x) = b$. These are $c_1$-continuous
$(b-a)$-maps such that
$f_a|_{\{a\}}$ is a 0-map, hence an $m$-map, and
$f_b|_{\{b\}}$ is a 0-map, hence an $m$-map.

\begin{itemize}
    \item If $b-a > m$, then $A$ is minimal, since 
          $f_a$ and $f_b$ are not $m$-maps.
    \item If $diam(X,c_1) = b-a \le m$, then every 
          member of $C(X,c_1)$ is an $m$-map, so
          $A$ is not minimal.
\end{itemize}
\end{proof}

We have the following.
\begin{prop}
\label{limitedImpliesCold}
Let $(X,\kappa)$ be a connected digital image 
that is $(A,m,n)$-limited
for some $\emptyset \neq A \subset X$ and 
$0 \le m \le n$. Then $A$ is an
$n$-cold set for $(X,\kappa)$.
\end{prop}

\begin{proof}
Let $f \in C(X,\kappa)$ such that $f|_A = \id_A$. Then $f|_A$ is a 
0-map, hence an $m$-map. Therefore, $f$ is an $n$-map. The assertion follows.
\end{proof}

\begin{prop}
\label{c2-coldSet}
{\rm \cite{BxFpSets2}}
Let $m,n \in \N$. Let
$X = [0,m]_{\Z} \times [0,n]_{\Z}$.
Let $A \subset Bd(X)$ be such that
no pair of $c_1$-adjacent members of $Bd(X)$ 
belong to $Bd(X) \setminus A$.
Then $A$ is a 1-cold set for $(X,c_2)$.
\end{prop}

Observe that the hypothesis for testing a set $A$ for $n$-coldness,
that $f|_A$ is a 0-map, is stricter than the hypothesis for testing 
$(A,m,n)$-limitedness, that $f|_A$ is an $m$-map. Therefore, it is perhaps
not surprising that the converse of Proposition~\ref{limitedImpliesCold}
is not generally true, as shown by the following.

\begin{exl}
\label{A(0,1)notA(1,1)}
Let $X = [0,2]_{\Z}^2$. Let 
\[ A = \{ \, (0,0), (0,2), (2,0), (2,2) \, \} \subset X.
\]
By Proposition~\ref{c2-coldSet}, $A$ is 1-cold for $(X,c_2)$, i.e.,
$(X,c_2)$ is $(A,0,1)$-limited. However,
$(X,c_2)$ is not $(A,1,1)$-limited.
\end{exl}

\begin{proof}
Let $f: X \to X$ be the function (see Figure~\ref{fig:2-map})
\[ f(x,y) = \left \{ \begin{array}{ll}
   (x,y) & \mbox{if } y > 0; \\
   (x,1) & \mbox{if } (x,y) \in \{ \, (0,0), (2,0) \, \}; \\
   (1,2) & \mbox{if } (x,y) = (1,0).
   \end{array} \right .
\]
It is easily seen that $f \in C(X,c_2)$ and $f|_A$ is a 1-map, but
$f$ is not a 1-map since $d_{(X,c_2)}((1,0), f(1,0)) = 2$.
\end{proof}

\begin{figure}
    \centering
    \includegraphics{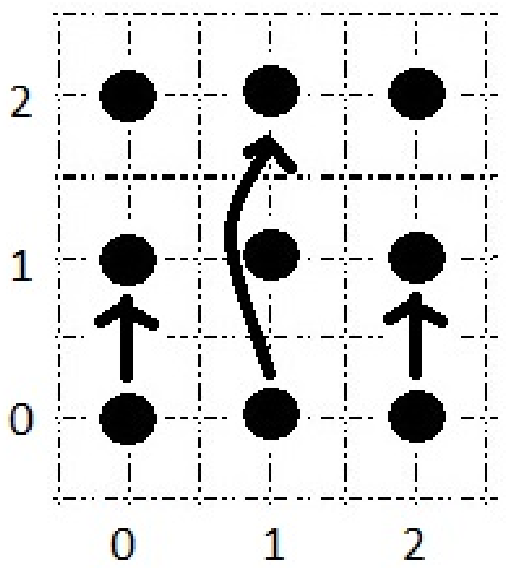}
    \caption{$X = [0,2]_{\Z}^2$. $f: X \to X$ is illustrated by
     arrows showing how points not fixed by $f$ are mapped. 
     \newline $f(0,0) = (0,1)$, ~~ $f(1,0) = (1,2)$, ~~ $f(2,0) = (2,1)$.
     \newline This is the function of Example~\ref{A(0,1)notA(1,1)}.
     One sees easily that $f \in C(X,c_2)$ and $f$ is a 2-map,
     although for the corner set $A$, a 1-cold set of $(X,c_2)$,
     $f|_A$ is a 1-map.
    }
    \label{fig:2-map}
\end{figure}

We show that being $(A,m,n)$-limited is an isomorphism invariant.
\begin{thm}
\label{invariant}
Let $F: (X,\kappa) \to (Y, \lambda)$ be an isomorphism of
connected digital images. Let
$\emptyset \neq A \subset X$. Let $m,n \in \N^*$.
If $(X,\kappa)$ is $(A,m,n)$-limited, then $(Y,\lambda)$
is $(F(A),m,n)$-limited. Further, if $A$ is a minimal
$(m,n)$-limiting set for $(X,\kappa)$, then
$F(A)$ is a minimal $(m,n)$-limiting set for 
$(Y, \lambda)$.
\end{thm}

\begin{proof}
Let $f \in C(Y,\lambda)$ such that $f|_{F(A)}$ is an $m$-map. 
By Theorem~\ref{composition}, $G=F^{-1} \circ f \circ F \in C(X,\kappa)$.

Let $a \in A$. Then 
\[ m \ge d_{(Y,\lambda)}(Fa, fFa) = 
   d_{(X,\kappa)}(F^{-1}Fa, F^{-1}fFa) =
   d_{(X,\kappa)}(a, Ga);
\]
i.e., $G|_A$ is an $m$-map. It follows that $G$ is an $n$-map.

Let $y = F(x)$ be an arbitrary point of $Y$. Then
\[ d_{(Y,\lambda)}(y, f(y)) = 
   d_{(X,\kappa)}(F^{-1}(y),F^{-1}(f(y))) =    \]
\[
   d_{(X,\kappa)}(F^{-1}(F(x)), F^{-1}(f(F(x)))) = 
   d_{(X,\kappa)}(x, G(x)) \le n.
\]
Thus $f$ is an $n$-map.

If $A$ is minimal, suppose $F(A)$ is not. Then
there is a proper subset $B$ of $F(A)$ such that
$B$ is an $(m,n)$-limiting set for $(Y,\lambda)$.
Then $F^{-1}(B)$ is a proper subset
of $A$ and, by the above, is 
$(m,n)$-limiting for $(X,\kappa)$.
This is a contradiction of the minimality of $A$,
which establishes that $B$ is minimal.
\end{proof}

\begin{thm}
\label{induceBy2k}
Let $\emptyset \neq A \subset X$ for a 
connected digital image
$(X,\kappa)$. Let $k \in \N^*$.
Suppose for every $x \in X$
there exists $a_x \in A$ such that $d_{(X,\kappa)}(x,a_x) \le k$.
Suppose $f \in C(X,\kappa)$. If
$f|_A$ is an $m$-map,
then $f$ is an $(m+2k)$-map.
\end{thm}

\begin{proof}
Given $x \in X$, let $f$ and $a_x$ be as described above. Then
\[ d_{(X,\kappa)}(x, f(x)) \le d_{(X,\kappa)}(x, a_x) + d_{(X,\kappa)}(a_x, f(a_x))
   + d_{(X,\kappa)}(f(a_x), f(x)) \le
\]
\[ k + m + k = m+2k.
\]
The assertion follows.
\end{proof}

\begin{cor}
\label{induceBy2k-cor}
Let $\emptyset \neq A \subset X$ for a 
connected digital image
$(X,\kappa)$. Let $k \in \Z$, $k \ge 0$.
Suppose for every $x \in X$
there exists $a_x \in A$ such that
$d_{(X,\kappa)}(x,a_x) \le k$.
Then $(X,\kappa)$ is 
$(A,m,m+2k)$-limited.
\end{cor}

\begin{proof}
This follows from Theorem~\ref{induceBy2k}.
\end{proof}

Recall $A$ is a {\em dominating set for} $(X,\kappa)$ if given $x \in X$
there exists $a \in A$ such that $x \adjeq_{\kappa} a$.

\begin{cor}
Let $A$ be a dominating set for a connected digital image
$(X,\kappa)$.
\begin{itemize}
    \item Suppose $f \in C(X,\kappa)$. If $f|_A$ is an $m$-map, 
          then $f$ is an $(m+2)$-map.
    \item If for every $f \in C(X,\kappa)$, $f|_A$ is an $m$-map, then
          $(X,\kappa)$ is $(A,m,m+2)$-limited.
\end{itemize}
\end{cor}

\begin{proof}
These assertions follow by taking $k=1$ in Theorem~\ref{induceBy2k}
and Corollary~\ref{induceBy2k-cor}.
\end{proof}

We cannot in general replace $m+2$ by $m+1$ in Theorem~\ref{induceBy2k}
or in Corollary~\ref{induceBy2k-cor}, as shown by the following.

\begin{exl}
Let $X = [-1,1]_{\Z}$. Then $A=\{ \, 0 \, \}$ dominates $(X, c_1)$ and
the function $f(x) = \, \mid x \mid$ is easily seen to be a member
of $C(X,c_1)$ such that $f|_A$ is a $0$-map, but $f$ is a $2$-map that
is not a $1$-map.
\end{exl}

\subsection{Diameter}
\begin{lem}
\label{distLemma}
Let $Y$ be a $\kappa$-connected subset of the
connected digital image $(X,\kappa)$. Let
$y_0,y_1 \in Y$. Then
\[ d_{(Y,\kappa)}(y_0,y_1) \ge d_{(X,\kappa)}(y_0,y_1).
\]
\end{lem}

\begin{proof}
We have $d_{(Y,\kappa)}(y_0,y_1) = length(P)$, where
$P$ is a shortest $\kappa$-path in~$Y$ from $y_0$
to $y_1$. Since $P \subset X$, the assertion follows.
\end{proof}

That the inequality in Lemma~\ref{distLemma} may
be strict is shown in the following.

\begin{exl}
Let $(X, \kappa)$, $X = \{ \, p_i \, \}_{i=0}^{m}$, be a
digital simple closed curve of $m+1$ points that are
circularly indexed, $m \ge 4$.
Let $Y = X \setminus \{ \, p_0 \, \}$. Then
\[ d_{(Y,\kappa)}(p_1,p_m) = m-1 > 2 =
   d_{(X,\kappa)}(p_1,p_m).
\]
\end{exl}

\begin{thm}
\label{diamThm}
    Let $(X,\kappa)$ be a finite connected digital
    image. Let $f: X \to X$ be an $m$-map. Consider
    the inequality
    \begin{equation}
        \label{diamCompare}
        diam(f(X)) \ge diam(X) - 2m
    \end{equation} 
    Inequality~{\rm (\ref{diamCompare})} is valid if
    $diam(X)$ is computed using $d_{(X,\kappa)}$
    or using $d_{(f(X),\kappa)}$.
\end{thm}

\begin{proof}
Since $X$ is finite, $diam(X) = M < \infty$.
Therefore, there exist $x,y \in X$ and a 
$\kappa$-path~$P$ in~$X$ of length $M$ from $x$ to $y$.
We have
    \[ M = d_{(X,\kappa)}(x,y) \le d_{(X,\kappa)}(x,f(x)) + d_{(X,\kappa)}(f(x),f(y)) + d_{(X,\kappa)}(f(y),y)
    \]
    \[
    \le m + d_{(X,\kappa)}(f(x), f(y)) + m, ~~~~~\mbox{or}
\]
\[ M - 2m \le d_{(X,\kappa)}(f(x),f(y)).
\]
Using Lemma~\ref{distLemma}, it follows that
\[  M - 2m \le diam_{(X,\kappa)}(f(X)) \le 
    diam_{(f(X),\kappa)}(f(X)).
\]
\end{proof}

\subsection{Hyperspace metrics}
The {\em Hausdorff metric}~\cite{Nadler} is often used 
as a measure of how similarly
positioned two objects are in a metric space. Given nonempty subsets
$A$ and $B$ of a metric space $(X,d)$, the Hausdorff metric (based on $d$)
for the distance between $A$ and $B$, $H_d(A,B)$, is the smallest
$\varepsilon \ge 0$ such that given $a \in A$ and $b \in B$, there
exist $a' \in A$ and $b' \in B$ such that
\[ \max \{ \, d(a,b'), d(a',b) \, \} \le \varepsilon.
\]

Borsuk's {\em metric of continuity}~\cite{Borsuk54, Borsuk67} based on a metric $d$ has been
adapted to digital topology~\cite{beyondHaus} as
follows. For digital images $Y_0$ and $Y_1$ in connected
$(X,\kappa)$, the metric of continuity $\delta_d(Y_0,Y_1)$ 
is the greatest lower bound of numbers
$t>0$ such that there are $\kappa$-continuous 
$f: Y_0 \to Y_1$ and $g: Y_1 \to Y_0$ with
\[ d_{(X,\kappa)}(x, f(x)) \le t \mbox{ for all } x \in Y_0 \mbox{ and } d_{(X,\kappa)}(y, g(y)) \le t 
    \mbox{ for all } y \in Y_1.
\]

\begin{prop}
\label{HausAndContinuityMetrics}
{\rm \cite{beyondHaus}}
Given finite digital images $(X,\kappa)$ and $(Y,\kappa)$ in $\Z^n$ and a metric $d$ for $\Z^n$,
$H_d(X,Y) \le \delta_d(X,Y)$.
\end{prop}

\begin{thm}
Let $(X,\kappa)$ be a finite connected digital image. 
Let $f \in C(X,\kappa)$ be an $m$-map. Then 
\[ H_{d_{(X,\kappa)}}(X,f(X)) \le
   \delta_{d_{(X,\kappa)}}(X,f(X)) \le m. 
\]
\end{thm}

\begin{proof}
That $H_{d_{(X,\kappa)}}(X,f(X)) \le \delta_{d_{(X,\kappa)}}(X,f(X))$ 
comes from Proposition~\ref{HausAndContinuityMetrics}.

Let $x \in X$ and let $y=f(x) \in f(X)$. By choice of $f$, $d_{(X,\kappa)}(x,y) \le m$.

Let $g: f(X) \to X$ be the inclusion function $g(y)= y$.
Clearly $g$ is a $0$-map, hence an $m$-map.
The assertion follows.
\end{proof}

\subsection{Rigid images}
\begin{defn}
\label{rigid}
{\rm \cite{HaarmanEtal}}
A digital image $(X,\kappa)$ is {\em rigid} if the only member of
$C(X,\kappa)$ that is homotopic to $\id_X$ is $\id_X$.
\end{defn}

See Figure~\ref{fig:rigid} for an example of a rigid digital image.

\begin{figure}
    \centering
    \includegraphics{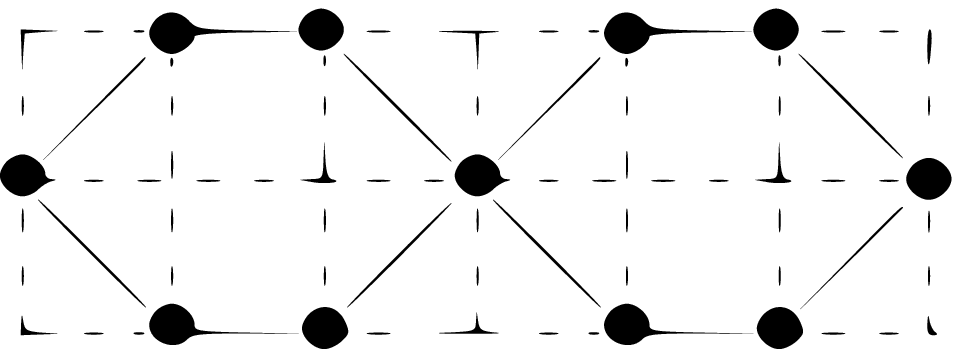}
    \caption{The digital image of Figure~1 of~\cite{bs20}. Example~3.11
    of~\cite{bs20} shows this image, using
    $c_2$ adjacency, is rigid.
    }
    \label{fig:rigid}
\end{figure}

\begin{prop}
\label{rigidThm}
Let $(X,\kappa)$ be a connected digital image that is rigid. Then
there is no $1$-map in $C(X,\kappa)$ other than $\id_X$.
\end{prop}

\begin{proof}
Let $f$ be a member of $ C(X,\kappa)$ that is a $1$-map. Then the function $h: X \times [0,1]_{\Z} \to X$ defined by
   \[ h(x,0) = x, ~~ h(x,1) = f(x)
   \]
is easily seen to be a homotopy from $\id_X$ to $f$. It
follows from Definition~\ref{rigid} that $f = \id_X$.
\end{proof}

\section{Trees}
In this section, $(T,\kappa)$ is a finite tree and
$A$ is the set of members of $T$ that have degree 1, i.e.,
for $a \in T$, $\#N(a)=1$ if and only if $a \in A$.

\begin{thm}
\label{treeFreeze}
{\rm \cite{BxFpSets2}}
$A$ is a minimal freezing set for $(T,\kappa)$
when $\#T > 1$.
\end{thm}

\begin{thm}
\label{tree1,1}
Let $(T,\kappa)$ be a finite tree and let $A$ be the 
set of members of $T$ that have degree 1. Then
$A$ is an $(m,m)$-limiting set for
$(T,\kappa)$, for $m \in \{ \, 0, 1 \, \}$. 
For $m=1$, $A$ is minimal if
and only if $\#T \neq 2$.
\end{thm}

\begin{proof}
Trivially, when $\#T = 1$, 
$A= \emptyset$ is $(0,0)$-limiting. 
By Theorem~\ref{treeFreeze}, $A$ is $(0,0)$-limiting 
when $\#T > 1$. Next we show that
$A$ is a $(1,1)$-limiting set for $(T,\kappa)$.

Let $x \in X$, $f \in C(T,\kappa)$ such that
$f|_A$ is a $1$-map. We show that the
assumption 
\begin{equation}
\label{treeAssump}
d_{(T,\kappa)}(x,f(x)) > 1
\end{equation}
leads to a contradiction.

There exist $y_0,y_1 \in A$ such that
the unique shortest $\kappa$-path $P$ in $T$ from
$y_0$ to $y_1$ contains $\{ \, x, f(x) \, \}$.

Let $p: [0,k]_{\Z} \to P$ be an isomorphism with
$p(0)=y_0$ and $p(k)=y_1$.
Since $y_0,y_1 \in A$, we must have
$\{ \, f(y_0), f(y_1) \, \} \subset P$. Since
$f|_A$ is a $1$-map, 
\[ f(y_0) \in \{ \, y_0 = p(0), p(1) \, \}, ~~~~~
   f(y_1) \in \{ \, y_1=p(k), p(k-1) \, \}
\]
Let
$x = p(t)$, $f(x) = p(t')$. 
\begin{itemize}
    \item Suppose $t \le t'$. By~(\ref{treeAssump}),
          $t' > t + 1$, and $d_{(T,\kappa)}(y_0,x) = t$.
          Then 
          \[ d_{(T,\kappa)}(f(x), f(y_0)) \ge t' - 1 >
             t+1-1 = t = d_{(T,\kappa)}(x,y_0),
          \]
          which is impossible, since $f$ is continuous.
    \item A similar contradiction arises if $t > t'$.
\end{itemize}
Therefore, $f$ is a 1-map, which establishes that
$A$ is $(1,1)$-limiting.

Suppose $\#T \neq 2$. As above,
$A$ is minimal for $\#T =1$. So assume $\#T > 2$.
Then $\#A > 1$, since 
$(T,\kappa)$ is a tree.
Let $a$ and $b$ be distinct members of $A$. Let
$P$ be the unique shortest $\kappa$-path in $T$ from
$a$ to $b$. There is a $(c_1,\kappa)$-isomorphism
$p: [0, length(P)]_{\Z} \to P$ such that $p(0)=a$.
Then the function $f: T \to T$ given by
\[ f(x) = \left \{ \begin{array}{ll}
    x  & \mbox{if } x \neq a; \\
    p(2) & \mbox{if } x = a,
   \end{array}    \right .
\]
is easily seen to belong to $C(T,\kappa)$; and $f$
is a 2-map. Thus $A \setminus \{ \, a \, \}$ is not
a $(1,1)$-limiting set, 
so $A$ is a minimal $(1,1)$-limiting set for $(T,\kappa)$.

For $\#T = 2$, $A \neq \emptyset$, so $A$ is not
minimal since in this case
we see easily that $\emptyset$ is a minimal
$(1,1)$-limiting set for $(T,\kappa)$.
\end{proof}

\section{Rectangles with axis-parallel sides}
In this section, we consider limiting sets for
digital rectangles of the form
$[a,b]_{\Z} \times [c,d]_{\Z} \subset \Z^2$.

\subsection{$c_1$ $(A,m,m)$-limited rectangles}
\begin{thm}
\label{c1Rectangles}
Let $X = [0,j]_{\Z} \times [0,k]_{Z}$. Let $A$ be 
the set of corners of $X$, i.e., 
\[ A = \{ \, (0,0),~(0,k),~(j,0),~(j,k) \, \}.
\]
Then $(X,c_1)$ is $(A,m,m)$-limited for all
$m \in \N^*$.
\end{thm}

\begin{proof}
Let $f \in C(X,c_1)$ such that $f|_A$ is an $m$-map.
Suppose there exists $(a,b) \in X$ such that
\begin{equation}
    \label{distFrOrigin}
    d_{(X,c_1)}((a,b),f(a,b)) > m
\end{equation}

First we consider the case in which $f$ moves $(a,b)$ 
away from the corner
$(0,0)$, i.e., $p_1(f(a,b)) \ge a$, $p_2(f(a,b)) \ge b$.
From~(\ref{distFrOrigin}) it follows that
\[  p_1(f(a,b)) - a + p_2(f(a,b)) - b > m, 
\]
so
\[ d_{(X,c_1)}((a,b),(0,0)) = a+b < 
   p_1(f(a,b)) + p_2(f(a,b)) - m \le 
\]
\[
    p_1(f(a,b)) + p_2(f(a,b)) - [p_1(f(0,0)) + p_2(f(0,0))]
\]
\[
    = d_{(X,c_1)}(f(a,b), f(0,0)).
\]
which is impossible since $f$ is continuous.

A similar argument can be made for each of the following cases.
\begin{itemize}
    \item $f$ moves $(a,b)$ away from the corner $(j,0)$, i.e.,
          \[p_1(f(a,b)) \le a,~~ p_2(f(a,b)) \ge b.
          \]
          Here we obtain
          $d_{(X,c_1)}(f(a,b), f(j,0)) > d_{(X,c_1)}((a,b), (j,0))$. As above,
          the latter is a contradiction.
    \item $f$ moves $(a,b)$ away from the corner $(0,k)$, i.e.,
          \[p_1(f(a,b)) \ge a,~~ p_2(f(a,b)) \le b.
          \]
          Here we obtain
          $d_{(X,c_1)}(f(a,b), f(0,k)) > d_{(X,c_1)}((a,b), (0,k))$. As above,
          the latter is a contradiction.
    \item $f$ moves $(a,b)$ away from the corner $(j,k)$, i.e.,
          \[p_1(f(a,b)) \le a,~~ p_2(f(a,b)) \le b.
          \]
          Here we obtain
          $d_{(X,c_1)}(f(a,b), f(j,k)) >
          d_{(X,c_1)}((a,b), (j,k))$. As above,
          the latter is a contradiction.
\end{itemize}

Since each case yields a contradiction, the assertion is established.
\end{proof}

\subsection{$c_n$ $(A,m,m)$-limited $n$-cubes}
\begin{thm}
\label{cnCube}
Let $X = \prod_{i=1}^n [a_i,b_i]_{\Z}$. Let $A$ be the set of corners of
$X$, i.e., $A = \prod_{i=1}^n \{ \, a_i,b_i \, \}$. Let $m > 0$.
Let $f \in C(X,c_n)$ such that $f|_A$ is $m$-limited. 
Then $f$ is $m$-limited. Hence $(X,c_n)$ is
$(A,m,m)$-limited.
\end{thm}

\begin{proof}
If otherwise, then for some $x = (x_1,\ldots, x_n) \in X$,
$d_{(X,c_n)}(x,f(x)) > m$. Then for some index~$i$,
$\mid x_i - p_i(f(x)) \mid ~ > m$. This leads to a contradiction as in
the proof of Theorem~\ref{c1Rectangles}. The assertion follows.
\end{proof}

\begin{prop}
\label{notCold}
{\rm \cite{BxFpSets2}}
Let $X = \prod_{i=1}^n [0,m_i]_{\Z} \subset \Z^n$, 
where $m_i > 1$ for all~$i$.
Let $A \subset Bd(X)$ be such that
$A$ is not $c_n$-dominating in $Bd(X)$.
Then $A$ is not a cold set for $(X,c_n)$.
\end{prop}

\begin{remark}
We cannot include 0 as a value of $m$ 
in Theorem~\ref{cnCube} since,
by Proposition~\ref{notCold},
if for some index $i$ we have $b_i - a_i > 2$ 
then $A$ is not cold and 
therefore is not a freezing set for $(X,c_n)$. 
\end{remark}

\section{Simple closed curves}
We consider limiting sets for digital simple closed
curves. Throughout this section, $C_v$ is a
digital simple closed curve (a {\em cycle} of $v$
points): $C_v = \{ \, x_i \, \}_{i=0}^{v-1}$ for 
some $v \ge 4$ with the members of $C_v$ 
indexed circularly.

For $v \ge 4$, we define
\[ D(v) = \left \{ \begin{array}{ll}
     \frac{v-2}{4} - 1 & \mbox{if } v-2 
                     \mbox{ is a multiple of } 4; \\
                     \\
     \lfloor \frac{v-2}{4} \rfloor & \mbox{if } v-2 
                     \mbox{ is not a multiple of } 4  .
     \end{array} \right .
\]

For $0 \le d < v$, we say the function
$r_d: C_v \to C_v$ given by
\[ r_d(x_i) = x_{(i+d) \mod v}
\]
is a {\em rotation}. 

The function $f: C_v \to C_v$ given by
$f(x_i) = x_{(v-i) \mod v}$ is the {\em flip map}.

Rotations and flip maps are isomorphisms of
$(C_v,\kappa)$.

\begin{thm}
{\rm \cite{bs20}}
\label{Cn-maps}
Let $f \in C(C_v,\kappa)$. Then 
\begin{itemize}
    \item $f$ is not surjective, or
    \item $f$ is a rotation, or
    \item $f$ is the composition of a flip map and a
          rotation.
\end{itemize}
\end{thm}

\begin{thm}
{\rm \cite{Bx10}}
\label{antipodes}
Let $f: C_v \to \Z$ be $(\kappa,c_1)$-continuous. 
Assume $v$ is even. If $f$ is not surjective,
there exists $x_u \in C_v$ such that
$f(x_u) \adjeq_{c_1} f(x_{u + v/2})$.
\end{thm}

\begin{prop}
\label{nonSurjectiveBound}
Let $f \in C(C_v,\kappa)$. Suppose $f$ is not
surjective and $v$ is even. Then there is an index~$j$
such that $d_{(C_v,\kappa)}(x_j,f(x_j)) \ge (v-2)/4$,
so $f$ is not a $D(v)$-map. 
\end{prop}

\begin{proof}
\label{nonsurjLim}
Since $f$ is nonsurjective and continuous,
$f(C_v)$ is $(\kappa,c_1)$-isomorphic to a
digital interval. By Theorem~\ref{antipodes},
there exists $x_u \in C_v$ such that
$f(x_u) \adjeq_{\kappa} f(x_{u + v/2})$. 
Hence
\[ v/2 = d_{(C_v,\kappa)}(x_u, x_{u+v/2}) \le \]
\[ d_{(C_v,\kappa)}(x_u, f(x_u)) +
   d_{(C_v,\kappa)}(f(x_u), f(x_{u+v/2})) +
   d_{(C_v,\kappa)}(f(x_{u+v/2}), x_{u+v/2}) \le \]
\[ d_{(C_v,\kappa)}(x_u, f(x_u)) + 1 +
   d_{(C_v,\kappa)}(f(x_{u+v/2}), x_{u+v/2})
\]
or
\[ \frac{v-2}{2} \le
    d_{(C_v,\kappa)}(x_u, f(x_u)) +
   d_{(C_v,\kappa)}(f(x_{u+v/2}), x_{u+v/2})
\]
Thus
\[ \max \{ \, d_{(C_v,\kappa)}(x_u,f(x_u)), 
        d_{(C_v,\kappa)}(x_{u + v/2},f(x_{u + v/2})) \, \}
    \ge  \frac{v-2}{4} > D(v).
\]
Thus $f$ is not a $D(v)$-map.
\end{proof}

\begin{prop}
\label{rotationLim}
The rotation $r_d$ is an $m_d$-map, for
$m_d=\min\{ \, d,v-d \, \}$. Indeed, 
$d_{(C_v,\kappa)}(x_i,r_d(x_i)) = m_d$ for all indices~$i$.
\end{prop}

\begin{proof}
Elementary and left to the reader.
\end{proof}

\begin{thm}
Let $A = \{ \, x_i, x_j, x_k \, \} \subset C_v$
where $C_v$ is a union of unique shorter arcs determined 
by pairs of these points. 
(Note that $A$ is a freezing set for $(C_v,\kappa)$,
by Theorem~\ref{3ptsForCycles}.) Let 
$R_{ij}$ be the unique shorter arc from $x_i$ to $x_j$,
$R_{ik}$ be the unique shorter arc from $x_i$ to $x_k$,
and
$R_{jk}$ be the unique shorter arc from $x_j$ to $x_k$.
Let 
\[ 0 \le m \le \min \left \{ \, D(v), ~length(R_{ij})/2,
   ~length(R_{ik})/2, ~length(R_{jk})/2 \, \right \}.
\]
Then $(C_v,\kappa)$ is $(A,m,m)$-limited.
\end{thm}

\begin{proof}
By our choice of $m$ and
Proposition~\ref{nonSurjectiveBound},
every $m$-map in $C(C_v,\kappa)$ is a surjection. 

By Proposition~\ref{rotationLim}, if $f \in C(C_v,\kappa)$
is a rotation and $f|_A$ is an $m$-map, then $f$ is
an $m$-map.

Let $\ell$ be the flip map of $C_v$,
$g = \ell \circ r_d: C_v \to C_v$, and
$g' = r_d \circ \ell: C_v \to C_v$.
By Theorem~\ref{composition}, $g,g' \in C(C_v,\kappa)$
and these functions are isomorphisms, hence
distance-preserving with respect to $d_{(C_v,\kappa)}$.
Let $G \in \{ \, g, \, g' \, \}$.

Suppose $G(x_i) \in N^*(x_i,m)$. By our choice of~$m$
and the fact that $G$ is orientation-reversing,
$G(x_j) \not \in N^*(x_j,m)$. Similarly, if
$G(x_j) \in N^*(x_j,m)$ or
$G(x_k) \in N^*(x_k,m)$, then
$G(x_i) \not \in N^*(x_i,m)$.
Therefore $G$ is not an $m$-map.

We conclude by Theorem~\ref{Cn-maps} that if
$f \in C(C_v,\kappa)$ such that $f|_A$ is an $m$-map,
then $f$ is a rotation and an $m$-map. Thus
$(C_v,\kappa)$ is $A(m,m)$-limited.
\end{proof}

\section{Limiting sets and retracts}
We show how if $X$ and $Y$ are finite connected
digital images such that $Y$ is a retract of $X$,
then a limiting set for $Y$ is a limiting set for
$X$, although not necessarily with the same
$(m,n)$ pair.

\begin{thm}
\label{extendFromRetract}
Let $\emptyset \ne A \subset Y \subset X$ such that
$(X,\kappa)$ and $(Y,\kappa)$ are finite and
connected and $Y$ is a $\kappa$-retract of $X$
by a retraction $r \in C(X,\kappa)$
that is an $\varepsilon$-map.
Let $h=H_{d_{(X,\kappa)}}(X,Y)$.
Suppose $Y$ is $(A, m, n)$-limited. 
Let $f \in C(X.\kappa)$ such that $f|_A$ is $m$-limited.
Then $f$ is an $(n + 2h + \varepsilon)$-map; hence
$(X,\kappa)$ is $(A, m, n + 2h + \varepsilon)$-limited.
\end{thm}

\begin{proof}
Let $f \in C(X,\kappa)$ be
such that $f|_A$ is an $m$-map. Then given $a \in A$,
there is a $\kappa$-path $P \subset X$ from $a$ to $f(a)$
of length at most $m$. Then
$r \circ f|_Y \in 
C(Y,\kappa)$ and $r(P)$ is a $\kappa$-path in $Y$
from $r(a)=a$ to $r(f(a))$ of length at most $m$.
Therefore, $r \circ f|_A$ is an $m$-map, so
$r \circ f|_Y$ is an $n$-map.

Let $x \in X$. 
There exists $y \in Y$ such that
$d(x,y) \le h$. Then
\[ d(x,f(x)) \le 
   d(x,y) + d(y, r(f(y))) +d(r(f(y)), r(f(x))) + d(r(f(x)), f(x))
\]
\[ \le h + n + d(y,x) + \varepsilon \le
   n + 2h + \varepsilon.
\] 
Thus, $f$ is an $(n + 2h + \varepsilon)$-map; thus
$(X,\kappa)$ is $(A, m, n + 2h + \varepsilon)$-limited.
\end{proof}

The bound $n +2h + \varepsilon$ in 
Theorem~\ref{extendFromRetract} is not generally
tight, as shown in the following.

\begin{exl}
\label{hairyRectangleExl}
Let $(X,\kappa)$ be a finited connected digital image.
Let $x_0 \in X$ and $A = Y = \{ \, x_0 \, \}$.
Let $h = H_{d(X,\kappa)}(X,Y)$. Clearly
$A$ is a freezing set for $(Y,\kappa)$, i.e., $(Y,\kappa)$
is $(A,0,0)$-limited. Since the function
$r: X \to X$ given by $r(x) = x_0$
is an an $h$-map and a retraction of $X$ to $Y$,
Theorem~\ref{extendFromRetract} implies
$(X, \kappa)$ is $(A,0,3h)$-limited.
However, $(X, \kappa)$ is $(A,0,h)$-limited, as
noted in Remark~\ref{elementaryProps}.
\end{exl}

\section{Cartesian products}
Elementary properties of limiting sets for Cartesian
products of digital images are discussed in this
section.

Given digital images or graphs $(X,\kappa)$
and $(Y,\lambda)$, the
{\em normal product adjacency} $NP(\kappa,\lambda)$ (also called the
{\em strong adjacency}~\cite{vLW}) generated by
$\kappa$ and $\lambda$ on the Cartesian product $X \times Y$ is defined
as follows.

\begin{defn}
\label{NP-def}
\rm{\cite{Berge,Sabidussi}}
Let $x, x' \in X$, $y, y' \in Y$.
Then $(x,y)$ and $(x',y')$ are $NP(\kappa,\lambda)$-adjacent in $X \times Y$
if and only if
\begin{itemize}
\item $x=x'$ and $y$ and $y'$ are $\lambda$-adjacent; or
\item $x$ and $x'$ are $\kappa$-adjacent and $y=y'$; or
\item $x$ and $x'$ are $\kappa$-adjacent and $y$ and $y'$ are $\lambda$-adjacent. \qed
\end{itemize}
\end{defn}

\begin{defn}
{\rm \cite{Sabidussi,BxNormal}}
Let $u,v \in \N$, $1 \le u \le v$.
Let $(X_i,\kappa_i)$ be digital images,
$i \in \{\,1,\ldots,v)$. Let $x_i,y_i \in X_i$,
$x = (x_1, \ldots, x_v)$, $y=(y_1,\ldots,y_v)$.
Then $x \adj y$ in the 
\emph{generalized normal product adjacency}
$NP_u(\kappa_1, \ldots, \kappa_v)$ if
for at least 1 and at most $u$ indices $i$,
$x_i \adj_{\kappa_i} y_i$, and for all other
indices $j$, $x_j=y_j$.
\end{defn}

Given a set of functions $f_i: X_i \to Y_i$ for
$1 \le i \le n$, the {\em product function} 
$\Pi_{i=1}^n f_i: \Pi_{i=1}^n X_i \to \Pi_{i=1}^n Y_i$
is the function
\[ \left ( \Pi_{i=1}^n f_i \right ) (x_1, \ldots, x_n) =
   (f_1(x_1), \ldots, f_n(x_n)), \mbox{ where }
   x_i \in X_i.
\]
\begin{thm}
\label{prod-cont}
{\rm \cite{BxNormal}}
Let $f_i: (X_i, \kappa_i) \to (Y_i, \lambda_i)$, $1 \le i \leq v$.
Then the product map
\[ \Pi_{i=1}^v f_i: (\Pi_{i=1}^v X_i, NP_v(\kappa_1, \ldots, \kappa_v)) \to (\Pi_{i=1}^v Y_i, NP_v(\lambda_1, \ldots, \lambda_v)) \]
is continuous if and only if each $f_i$ is continuous.
\end{thm}

\begin{thm}
\label{prodOnMaps}
Let $X = \prod_{i=1}^v X_i$. Let 
$\kappa=NP_v(\kappa_1, \ldots, \kappa_v)$.
Let $\emptyset \neq A \subset X$. Let
$A_i = p_i(A)$ for each index~$i$. Suppose 
$(X,\kappa)$ is $(A,m,n)$-limited.
Then for each index~$i$, $(X_i,\kappa_i)$ is
$(A_i,m,n)$-limited.
\end{thm}

\begin{proof}
Suppose $f_i \in C(X_i, \kappa_i)$ such that
$f_i|_{A_i}$ is an $m$-map. Then for each $x_i \in A_i$
there is a path $g_i: [0,m]_{\Z} \to X_i$ 
from $x_i$ to $f_i(x_i)$.
Consider the function
$G: [0,m]_{\Z} \to X$ given by
\[ G(t) = (g_1(t), \ldots, g_v(t)).
\]
Clearly, $G$ is $(c_1,\kappa)$-continuous.
Also, by Theorem~\ref{prod-cont},
\[ F = \prod_{i=1}^v f_i \in C(X,\kappa).
\]
Thus, $G$ is a $\kappa$-path
of length at most $m$ from
$(x_1,\ldots,x_v) \in A$ to
$F(x_1, \ldots, x_v)$. Hence
$F|_A$ is an $m$-map.
Therefore, $F$ is an $n$-map.
Therefore, given
$x=(x_1,\ldots,x_v) \in X$, there is a path 
$H: [0,n]_{\Z} \to X$ from $x$ to $F(x)$. Then
$p_i \circ H$ is a $\kappa_i$-path of length at 
most $n$ in $X_i$ from $x_i$ to $f_i(x_i)$. 
The assertion follows.
\end{proof}

\section{Infinite $X$}
\label{infX}
In this section, we give elementary properties of
limiting sets for infinite digital images.

\begin{prop}
Let $X$ be an infinite subset of $\Z^n$. Then
for $1 \le u \le v$, any $m,n \in \N^*$, and any 
finite $A \subset X$, $(X,c_u)$ is not $(A,m,n)$ limited.
\end{prop}

\begin{proof}
Given any finite subset $A$ of $\Z^n$, there exists a cube
$Y = [-k,k]_{\Z}^v$ such that $A \subset Y$. 

Let $r_1: \Z \to \Z$ be the $c_1$-retraction
\[ r_1(x) = \left \{ \begin{array}{ll}
    -k & \mbox{if } x \le -k; \\
     x & \mbox{if } -k \le x \le k; \\
     k & \mbox{if } k \le x.
   \end{array} \right .
\]
Then the function $r_v: X \to Y$ given by
\[ r_v(x_1,\ldots, x_v) = (r_1(x_1), \ldots, r_1(x_v))
\]
is easily seen to be a $c_u$-retraction of $\Z^n$ to $Y$.
Therefore $r|_A$ is a $0$-map, hence an $m$-map.

However, since $X$ is infinite and $Y$ is finite,
given any $m \in \N^*$, there exists $x \in X$ such
that no member of $Y$ is within $m$ of $x$ in the
$d_{(X,c_u)}$ metric. The assertion follows.
\end{proof}

\begin{exl}
Let $m \in \N$. The set 
\[ m\Z = \{ \, x \in \Z \mid x=mk \mbox{ for some }
           k \in \Z \, \}
\]
is an $(m,m)$-limiting set for $(\Z,c_1)$.
\end{exl}

\begin{proof}
If $m=1$ then the assertion is trivial. Thus we
assume $m>1$.

Let $f \in C(\Z,c_1)$ such that $f|_{m\Z}$ is
an $m$-map. Let $z \in \Z \setminus m\Z$. For some
$k \in Z$ and $q \in [1,m-1]_{\Z}$,
\[ z = mk + q.
\]
We must show $\mid f(z) - z \mid ~ \le m$. Suppose this
is false.

If $k \ge 0$ we proceed as follows.
\begin{itemize}
    \item If $f(z) - z > m$ then
          \[ f(z) > z + m = m(k+1) + q,
          \]
           and, since $f|_{m\Z}$ is an $m$-map,
          $f(mk) \le m(k+1)$, so
         \[ f(z) - f(mk) > m(k+1) + q - m(k+1)  = q.
         \]
         Thus $f([mk,z]_{\Z})$ is a $c_1$- path in $\Z$
         of length greater than $q$. But $[mk,z]_{\Z}$ is a $c_1$-path in $\Z$ of length $q$.
         This is impossible.
    \item Similarly, if $f(z) - z < -m$, we obtain
          a contradiction.
\end{itemize}

Similarly, if $k < 0$, we must obtain a contradiction.
It follows that $f$ is an $m$-map. Hence $m\Z$
is an $(m,m)$-limiting set for $(\Z,c_1)$.
\end{proof}

\section{Further remarks}
The fixed point theory for digital images has led us
to the study of limiting sets, in the sense that
the notion of an $(A,m,n)$-limited digital image
$(X,\kappa)$ generalizes the notion of $A$ being a
freezing set for $(X,\kappa)$. If $m$ and $n$ are
small relative to $diam(X,\kappa)$, we may expect
$f \in C(X,\kappa)$, such that $f|_A$ is an $m$-map,
to move no point of $X$ by very much (i.e., by 
more than $n$) and therefore, $f(X)$ might be
expected to resemble $X$ (although such a conclusion will
be subjective and may subjectively admit of exceptions).

We have explored several basic properties of
limiting sets, including some of their relationships
with retractions, Cartesian products, and infinite
cardinality. We have seen that often, 
if $A$ is a freezing set for $(X,\kappa)$ (i.e.,
$(X,\kappa)$ is $(A,0,0)$-limited),
then $(X,\kappa)$ is $(A,m,n)$-limited for small $m,n$ 
such that $(m,n) \neq (0,0)$.

We are grateful to a
reviewer for corrections
and suggestions.

\end{document}